\documentclass[11pt, abstract]{scrartcl}
\usepackage{amsmath}
\usepackage{amssymb}
\usepackage{amsthm}
\usepackage{physics}
\usepackage{units}
\usepackage[final]{microtype}
\usepackage{caption}
\usepackage{subcaption}
\usepackage{booktabs}
\usepackage{datatool}
\usepackage{multirow}
\usepackage{xcolor}
\usepackage{tikz}
\usetikzlibrary{arrows.meta, external, calc}
\usepackage{pgfplots}
\pgfplotsset{compat=newest}
\usepackage[backend=biber,hyperref,backref,bibencoding=utf8,defernumbers=true,style=numeric-comp,giveninits=true,maxnames=99,doi=false,isbn=false]{biblatex}
\renewbibmacro{in:}{\ifentrytype{article}{}{\printtext{\bibstring{in}\intitlepunct}}}
\usepackage{csquotes}
\makeatletter
\let\blx@rerun@biber\relax
\makeatother

\usepackage{hyperref}
\usepackage{cleveref}

\theoremstyle{plain}
	\newtheorem{theorem}{Theorem}
	\newtheorem{lemma}[theorem]{Lemma}
\theoremstyle{definition}
	\newtheorem{definition}[theorem]{Definition}
	
	\newtheorem{remark}[theorem]{Remark}

\renewcommand{\vec}[1]{{\mathbf{#1}}}
\newcommand{\jump}[1]{[\![{#1}]\!]}
\newcommand{\avg}[1]{\{\!\!\{{#1}\}\!\!\} }

\newcommand{\dist}{{\operatorname{dist}}}

\newcommand{\R}{{\mathbb R}}
\newcommand{\N}{{\mathbb N}}
\renewcommand{\P}{{\mathbb P}}

\newcommand{\rti}{{I_h^{\text{RT}}}}
\newcommand{\bdmih}{{I_h^{\text{BDM}}}}
\newcommand{\Rti}{{\mathbf{I}_h^{\text{RT}}}}
\newcommand{\hdivi}{{I_h^{\vec{H}(\operatorname{div})}}}
\newcommand{\cri}{{I_h^{\text{CR}}}}
\newcommand{\li}{{I_h^{\text{L}}}}
\newcommand{\fori}{{I_h^{\text{F}}}}
    
\newcommand{\Faces}{{\mathcal{F}(\mathcal{T}_h)}}    
\newcommand{\innerFaces}{{\mathcal{F}^i(\mathcal{T}_h)}}  
\newcommand{\boundaryFaces}{{\mathcal{F}^b(\mathcal{T}_h)}} 
\newcommand{\Hdiv}{{\vec{H}(\operatorname{div},\Omega)}}
\newcommand{\BDM}{{\text{BDM}(\mathcal{T}_h)}}
\newcommand{\RT}{{\text{RT}(\mathcal{T}_h)}}
\newcommand{\MAC}{{\mathit{MAC}}}
\newcommand{\RVP}{{\mathit{RVP}}}

\newcommand{\abssec}[1]{\noindent\small {\bfseries #1\quad}\ignorespaces}
\renewenvironment{abstract}{\abssec{Abstract}}{\par\vspace{1em}}
\newenvironment{keywords}{\abssec{Keywords}}{\par\vspace{1em}}
\newenvironment{MSC}{\abssec{Mathematics Subject Classification (2010)}}{\par\vspace{1em}}

\title{A nonconforming pressure-robust finite element method for the Stokes equations on anisotropic meshes}
\date{\today}
\author{Thomas Apel \and Volker Kempf \and Alexander Linke \and Christian Merdon}

\begin{document}
\maketitle

\begin{abstract}
	Most classical finite element schemes for the (Navier--)Stokes equations are neither pressure-robust, nor are they inf-sup stable on general anisotropic triangulations. 
	A lack of pressure-robustness may lead to large velocity errors, whenever the Stokes momentum balance is dominated by a strong and complicated pressure gradient. 
	It is a consequence of a method, which does not exactly satisfy the divergence constraint. 
	However, inf-sup stable schemes can often be made pressure-robust just by a recent, modified discretization of the exterior forcing term, using $\vec{H}(\operatorname{div})$-conforming velocity reconstruction operators. 
	This approach has so far only been analyzed on shape-regular triangulations. 
	The novelty of the present contribution is that the reconstruction approach for the Crouzeix--Raviart method, which has a stable Fortin operator on arbitrary meshes, is combined with results on the interpolation error on anisotropic elements for reconstruction operators of Raviart--Thomas and Brezzi--Douglas--Marini type, generalizing the method to a large class of anisotropic triangulations. 
	Numerical examples confirm the theoretical results in a 2D and a 3D test case.	
\end{abstract}
\begin{keywords}	
	anisotropic finite elements, incompressible Navier-Stokes equations, divergence-free methods, pressure-robustness
\end{keywords}
\begin{MSC} 
	65N30, 65N15, 65D05    
\end{MSC}

\section{Introduction}
Classical finite element methods for the incompressible Navier--Stokes equations, e.g. the Taylor--Hood family of finite elements, typically do not yield exactly divergence free solutions in the sense of $\vec{H}(\operatorname{div})$, but instead relax the divergence constraint in order to achieve discrete inf-sup stability \cite{LedererLinkeMerdonSchoberl2017}. 
The resulting error estimates for $\vec{H}^1$-conforming methods for the Stokes equations
\begin{align*}
	-\nu \Delta \vec{u} + \nabla p &= \vec{f}, \\
	\nabla \cdot \vec{u} &= 0,
\end{align*}
are of the form, see e.g. \cite{JohnLinkeMerdonNeilanRebholz2017,GiraultRaviart1986},
\begin{equation}
	\norm{\vec{u}-\vec{u}_h}_{1} \leq 2(1+C_F) \inf_{\vec{v}_h\in\vec{X}_h} \norm{\vec{u}-\vec{v}_h}_{1} + \frac{1}{\nu} \inf_{q_h\in Q_h} \norm{p - q_h}_{0},
\end{equation}
i.e. the quality of the velocity estimate depends on the pressure and possibly deteriorates unboundedly for $\nu \to 0$ posing a classical locking phenomenon in the sense of Babu\v{s}ka and Suri \cite{BabuskaSuri1992}.
We remark that the constant $C_F$ in the estimate denotes the stability constant of the Fortin operator of the mixed method.
On the other hand, exactly divergence-free $\vec{H}^1$ or $\vec{H}(\operatorname{div})$ conforming methods of order $k$, see e.g. \cite{Nedelec1980,Nedelec1986,BrezziDouglasMarini1985,RaviartThomas1977,ScottVogelius1985,SchroederLehrenfeldLinkeLube2018}, produce error estimates of the type
\begin{equation}
	\norm{\vec{u}-\vec{u}_h}_{1,h} \leq C_F \inf_{\vec{v}_h\in\vec{X}_h} \norm{\vec{u}-\vec{v}_h}_{1,h} + C h^k\abs{\vec{u}}_{k+1},
\end{equation}
which provide a much better control on the velocity error, independent of the pressure approximability. 

These methods have been known since the 1980s, see e.g. \cite{Nedelec1980,Nedelec1986,ScottVogelius1985}, and have significant advantages, especially in settings where the viscosity parameter $\nu$ is small or where the pressure approximation in the discrete pressure space is of low order. 
However they were not in the focus for practical applications where incompressible flows needed to be computed on a large scale, which was mainly due to two reasons: their more complicated implementation compared to the classical methods and their higher computational cost. 
Both issues are being addressed in current research: highly automated finite element libraries like NGSolve \cite{Schoberl2014} and FEniCS \cite{LoggMardalWells2012} offer a large choice of available elements, and the computational cost can be decreased significantly, e.g. by hybridization, see \cite[Appendix]{SchroederLehrenfeldLinkeLube2018}. 

Another way to get to a pressure-robust discretization has been introduced recently, see \cite{Linke2014}. It uses a reconstruction operator for the velocity test functions to reestablish $L^2$ orthogonality between the test functions and the irrotational part of the Helmholtz decomposition of the external force in the Stokes case, which results in regaining pressure-robustness for standard methods. The approach was first used on the Crouzeix--Raviart element, but has been applied to several other classical elements, see \cite{LinkeMatthiesTobiska2016,LinkeMerdon2016,LinkeMerdonWollner2017}. 

Unfortunately, all these results have in common that they assume a shape-regular triangulation of the domain. This assumption is in general not valid in practical applications, as incompressible flows tend to form boundary and interior layers, and in these regions adaptive mesh refinement strategies lead to highly stretched elements. 

However, there are a couple of established finite element methods, where a uniform stability of the Fortin operator has been shown on anisotropic mesh families. 
By classical mixed theory, this leads to a uniform inf-sup stability in the anisotropic case as well, which is needed for the discrete pressure error estimates. 
The most remarkable example is the nonconforming Crouzeix--Raviart element \cite{ApelNicaiseSchoberl2001,ApelNicaiseSchoberl2001:2}, where the stability constant of its Fortin operator is $C_F = 1$ on \emph{arbitrary} simplex grids, including anisotropic elements, evidently. 
See \Cref{lem:CRFortin} for the detailed result.
Further elements, which have shown to be applicable on anisotropic grids comprise the Bernardi--Raugel element in 2D and related elements \cite{ApelNicaise2004} and nonconforming rectangular elements \cite{ApelMatthies2008}, all combined with discontinuous pressure approximations. 
There are also results for the $hp$-version finite element method \cite{AinsworthCoggins2000,AinsworthCoggins2002,SchotzauSchwabStenberg1999} and recently, for certain anisotropic triangulations, the lowest order Taylor--Hood elements \cite{BarrenecheaWachtel2019}. 
As mentioned before, these discretizations are not pressure-robust.

We address the question of uniformly stable, pressure-robust methods for anisotropic grids in the present contribution, by combining the approach for pressure-robustness from \cite{Linke2014,BrenneckeLinkeMerdonSchoberl2015} and the error estimates for Raviart--Thomas and Brezzi--Douglas--Marini interpolation from \cite{AcostaApelDuranLombardi2011,ApelKempf2019}. 
In particular, we focus on two relaxations of the usual minimum angle condition on the shape of the elements. 
We consider triangles and tetrahedra which satisfy a maximum angle condition or additionally a regular vertex condition. 
The maximum angle condition was first introduced in \cite{Synge1957} for triangles and generalized for tetrahedra in \cite{Krizek1992}, and is frequently used, see e.g. \cite{AcostaApelDuranLombardi2011,AcostaDuran1999,Apel1999,DuranLombardi2008}. 
It is satisfied if all angles inside an element are bounded away from $\pi$. The regular vertex condition on the other hand is satisfied, when there is a vertex for which the outgoing unit vectors along the edges are uniformly linearly independent. 
We give proper definitions in \Cref{sec:preliminaries}. 

In two dimensions, the regular vertex property is trivially satisfied if the maximum angle condition is met, but the conditions are not equivalent in three dimensions. 
This becomes relevant for anisotropic meshes that arise when handling singularities near concave edges of the domain, see e.g. \cite{ApelNicaiseSchoberl2001} and the arguments in \Cref{subsec:aniso_meshes}.

For triangulations of both types, we prove optimal error estimates for convex domains and full elliptic regularity and show numerical experiments which support the theoretical results. The main result of this article is the generalization of the results from \cite{BrenneckeLinkeMerdonSchoberl2015,Linke2014} to a more general class of meshes, which requires only the maximum angle condition, and some sharper estimates under the assumption of a regular vertex property, allowing the method to be used on more application-oriented meshes.

In \Cref{sec:preliminaries} we introduce the required notation, some aspects of anisotropic triangulations and the continuous and discrete setting of the Stokes equations.
In \Cref{sec:CR_results} we recall some properties of the Crouzeix--Raviart element, which make it favorable to use for anisotropic settings.
\Cref{sec:a_priori_errors} contains the main results, the a-priori error estimates for the Stokes problem without the constraint of shape-regular triangulations.
The numerical examples are presented in \Cref{sec:numerical_examples}.

\section{Preliminaries}\label{sec:preliminaries}
\subsection{Notation}
Throughout the text we use bold symbols for vectors, vector-valued functions and their function spaces. 
The symbol $C$ denotes a generic constant which may change from line to line. When writing volume and surface integrals, we usually omit the integration measure, where the meaning is clear. 

By $\mathcal{T}_h$ we denote a conforming simplicial triangulation of the considered domain $\Omega \subset \R^d$, where $d\in \{2,3\}$ is the space dimension. 
The global mesh size parameter is defined by 
\begin{equation*}
	h = \max_{T\in\mathcal{T}_h} h_T,
\end{equation*} 
where $h_T$ is the diameter of the element $T\in \mathcal{T}_h$. 
By $\Faces$ we denote the set of all simplex facets of the triangulation $\mathcal{T}_h$, i.e. depending on $d$ the edges of triangles or faces of tetrahedra, and by $\innerFaces$ the set of all interior facets. 
For an element $T\in\mathcal{T}_h$, $\mathcal{F}(T) \subset \Faces$ is the set of all facets of $T$. 
For an element $T\in\mathcal{T}_h$ and a facet $F\in \Faces$ we denote by $\vec{x}_T$ and $\vec{x}_F$ their barycenters, respectively. 
For any facet $F\in\Faces$ let $\vec{n}_F$ denote its unit normal vector, which is oriented outward for boundary facets $F\in\boundaryFaces=\Faces\setminus\innerFaces$ and has an arbitrary but fixed orientation for interior facets. 
When considering a facet $F\in\mathcal{F}(T)$ of an element $T\in \mathcal{T}_h$, $\vec{n}_{F_T}$ denotes the outward facing normal vector with respect to the element. 
The aspect ratio $\sigma_T$ of an element $T\in\mathcal{T}_h$ is defined as
\begin{equation*}
	\sigma_T = \frac{h_T}{\rho_T},
\end{equation*}
where $\rho_T$ is the supremum of the diameters of all spheres contained in $T$. We denote by $\sigma$ the maximum of the occurring aspect ratios in the triangulation.

\subsection{Anisotropic meshes}\label{subsec:aniso_meshes}
When dealing with the relaxed notion of anisotropic triangulations, i.e. do not set an upper bound for the triangulation's aspect ratio, still some regularity is required of the elements. We define two such conditions.
\begin{definition}\label{def:maximum_angle_condition}
	An element $T$ satisfies the \emph{maximum angle condition} with a constant $\bar{\phi} < \pi$, written as $\MAC(\bar{\phi})$, if the maximum angle between facets and, for $d=3$, the maximum angle inside the facets are less than or equal to $\bar{\phi}$. A triangulation satisfies $\MAC(\bar{\phi})$, if all elements do.
\end{definition}
The maximum angle condition for triangles was first used in \cite{Synge1957}, and generalized to tetrahedra in \cite{Krizek1992}. It is very common when dealing with anisotropic elements, see e.g. \cite{AcostaApelDuranLombardi2011,AcostaDuran1999,Apel1999,DuranLombardi2008}. The next property is equivalent to the maximum angle condition for $d=2$, see \cite[Section 5, p. 29]{AcostaDuran1999}, while in three dimensions it describes a proper subclass. 
\begin{definition}\label{def:regular_vertex_property}
	An element $T$ satisfies the \emph{regular vertex property} with a constant $\bar{c}$, written as $\RVP(\bar{c})$, if there is a vertex $\vec{p}_{T,k}$ of $T$, so that for the matrix $N_k$, made up of the unit column vectors $\vec{l}_{T,j}^k=\frac{\vec{p}_{T,j}-\vec{p}_{T,k}}{\norm{\vec{p}_{T,j}-\vec{p}_{T,k}}}$ outgoing from vertex $\vec{p}_{T,k}$ towards vertex $\vec{p}_{T,j}$, $j\in \{1,\ldots,d+1\}\setminus\{k\}$, the inequality 
	\begin{equation*}
		|\det N_k| \geq \bar{c} > 0
	\end{equation*}
	holds. The vertex $\vec{p}_{T,k}$ is then called \emph{regular vertex} of the element $T$. Without loss of generality for the rest of the text we assume that the vertices are numbered so that $\vec{p}_{T,d+1}$ is the regular vertex, so that we can use the more intuitive notation $\vec{l}_{T,i} = \vec{l}_{T,i}^{d+1}$ and the element size parameters $h_{T,i}$, $i\in \{1,\ldots,d\}$, which are defined as the lengths of the edges corresponding to the vectors $\vec{l}_{T,i}$. 
\end{definition}

\begin{figure}[t]
	\centering
	\includegraphics[scale=1]{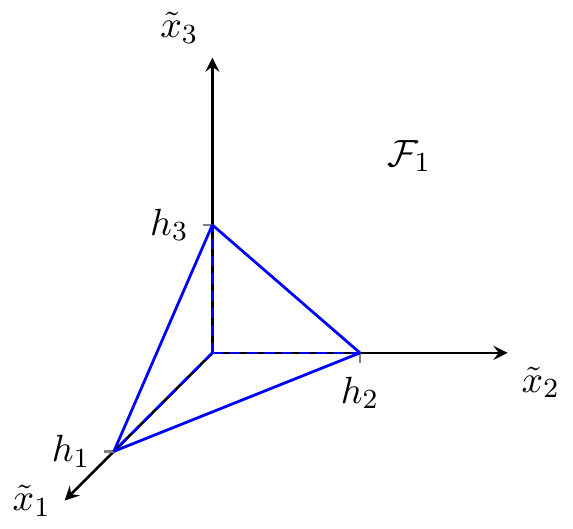}
	\includegraphics[scale=1]{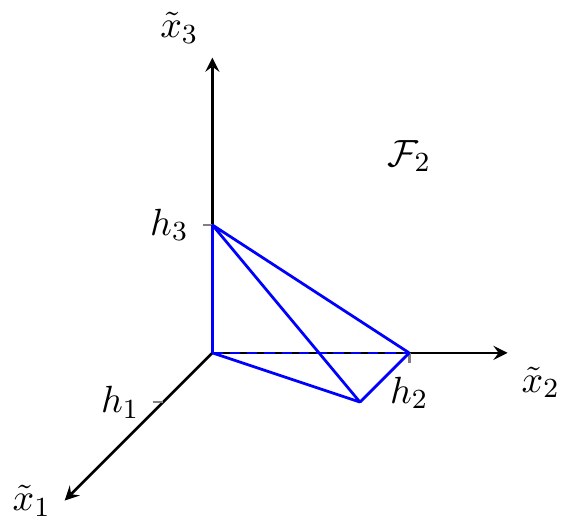}
	\caption{Families $\mathcal{F}_1$ and $\mathcal{F}_2$ of tetrahedra satisfying the maximum angle condition (left and right), and the regular vertex condition (left), Figure from \cite{ApelKempf2019}}\label{fig:Reference_families}
\end{figure}
As proved in \cite[Theorem 2.2, Theorem 2.3]{AcostaApelDuranLombardi2011}, the families $\mathcal{F}_1$ and $\mathcal{F}_2$ of elements pictured in \Cref{fig:Reference_families}, with arbitrary size parameters $h_i$, are sufficient to get any tetrahedron satisfying $\RVP(\bar{c})$ or $\MAC(\bar{\phi})$, using $\mathcal{F}_1$ or $\mathcal{F}_1 \cup \mathcal{F}_2$ respectively, by a reasonable affine transformation $F$, i.e. $F({\tilde{\vec{x}}}) = J_T  \tilde{\vec{x}} + \vec{x}_0$, $J_T\in \R^{d\times d}$, where $\norm{J_T}_{\infty}, \norm{J_T^{-1}}_{\infty} \leq C$, with $C$ only dependent on $\bar{\phi}$ resp. $\bar{c}$.

As described in \cite{ApelLube1995}, depending on the type of anisotropy in the elements, it may not be possible to fill arbitrary volumes with tetrahedra satisfying the regular vertex property, and the second type of reference family needs to be considered. 
Observe for example the three tetrahedra resulting from the subdivision of a triangular prism, as seen in  \Cref{fig:Pentahedron}. 
Two of those, $\vec{p}_1\vec{p}_2\vec{p}_3\vec{p}_6$ and $\vec{p}_1\vec{p}_4\vec{p}_5\vec{p}_6$ in the figure, clearly satisfy the regular vertex property, independent of the anisotropy of the prism. 
Now suppose the prism is stretched in $x_3$ direction, i.e. $h_3 \gg h_1,h_2$, as might be the case when grading the mesh towards a singular edge, see e.g. \cite{ApelNicaiseSchoberl2001,ApelNicaiseSchoberl2001:2}, then the remaining tetrahedron $\vec{p}_1\vec{p}_2\vec{p}_5\vec{p}_6$ does not satisfy the regular vertex property. 
If on the other hand we grade the mesh towards a boundary layer, e.g. $h_1\sim h_2 \gg h_3$, the third tetrahedron has a flat instead of a long shape, and the property is satisfied.

\begin{figure}[t]
	\centering
	\includegraphics[scale=1]{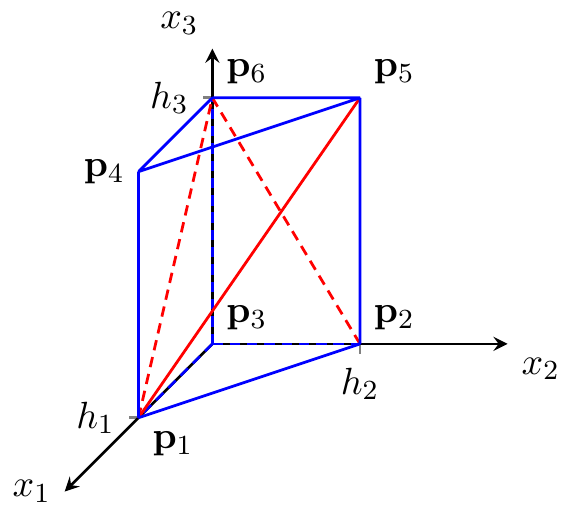}
	\caption{Subdivision of triangular prism in three tetrahedra}\label{fig:Pentahedron}
\end{figure}

For the rest of the text, except where explicitly stated, we assume that all triangulations at least satisfy a maximum angle condition.

\subsection{The continuous setting}
We consider the steady state incompressible Stokes equations in a simply connected, polyhedral domain $\Omega\subset\R^d$, $d\in\{2,3\}$, with homogeneous Dirichlet boundary conditions and external forcing $\vec{f}\in \vec{L}^2(\Omega)$ in the form
\begin{equation}\label{eq:stokescont}
\begin{split}
	-\nu \Delta\vec{u} + \nabla p &= \vec{f} \quad \text{on } \Omega,	\\
	\nabla \cdot \vec{u} &= 0 \quad \text{on } \Omega,	\\
	\vec{u} &= 0 \quad \text{on } \partial \Omega.
\end{split}
\end{equation}
Employing the function spaces
\begin{align*}
	&\vec{X} = \vec{H}^1_0(\Omega) = \{\vec{v} \in \vec{H}^1(\Omega): \vec{v} = 0 \text{ on } \partial \Omega\},	\\
	&Q = L^2_0(\Omega),
\end{align*}
the weak formulation of the problem is given as, see \cite[Section I.5.1]{GiraultRaviart1986}: find $(\vec{u},p) \in \vec{X}\times Q$, so that
\begin{equation}\label{eq:StokesContWeak}
\begin{split}
	a(\vec{u},\vec{v}) + b(\vec{v},p) &= l(\vec{v}),	\\
	b(\vec{u},q) &= 0,
\end{split}
\end{equation}
holds for all $(\vec{v},q)\in\vec{X}\times Q$, where the bilinear and linear forms are defined by
\begin{alignat*}{2}
	&a:\vec{X}\times\vec{X}\rightarrow \R,&&\quad a(\vec{u},\vec{v}) = \nu\int_\Omega \nabla\vec{u}:\nabla\vec{v},	\\
	&b:\vec{X}\times Q \rightarrow \R,&&\quad b(\vec{v},q) = -\int_\Omega q\nabla \cdot \vec{v}, \\
	&l:\vec{X} \rightarrow \R,&&\quad l(\vec{v}) = \int_\Omega \vec{f}\cdot\vec{v}. %
\end{alignat*}

With the space $\vec{V}^0 = \{\vec{v}\in \vec{X} : \nabla\cdot \vec{v} = 0\}$ of functions satisfying the divergence constraint, we can reformulate the problem in the elliptic form, see \cite[Section I.5.1]{GiraultRaviart1986}: find $\vec{u}\in \vec{V}^0$, so that
\begin{equation} \label{eq:StokesEllipticContinuous}
	a(\vec{u},\vec{v}) = l(\vec{v})
\end{equation}
holds for all $\vec{v}\in\vec{V}^0$. For the Stokes problem the continuous inf-sup condition
\begin{equation*}
	\exists \beta > 0: \quad \inf_{q\in Q}\sup_{\vec{v}\in\vec{X}} \frac{b(\vec{v},q)}{\norm{\vec{v}}_\vec{X} \norm{q}_0} \geq \beta
\end{equation*}
holds, see \cite[Section I.5.1]{GiraultRaviart1986}, where $\norm{\cdot}_k$ denotes the norm of the Sobolev space $H^k(\Omega)$ for $k\geq 0$.

\subsection{The discrete setting and interpolation operators}
For our method, we need some tools from the discontinuous Galerkin framework. We denote by $\jump{\vec{v}}_F$ and $\avg{\vec{v}}_F$ the jump and average, respectively, of a piecewise $\vec{H}^1$ function $\vec{v}$ over a facet $F$, which, see e.g. \cite[Section 1.2.3]{DiPietroErn2012}, are defined for an interior facet $F$ belonging to two elements $T_1$ and $T_2$ by
\begin{align*}
	\jump{\vec{v}}_F(\vec{x}) &= \vec{v}|_{T_1}(\vec{x}) - \vec{v}|_{T_2}(\vec{x}), \\
	\avg{\vec{v}}_F (\vec{x}) &= \frac{1}{2}(\vec{v}|_{T_1}(\vec{x}) + \vec{v}|_{T_2}(\vec{x})).
\end{align*}
For boundary faces we use the convention $\jump{\vec{v}}_F = \avg{\vec{v}}_F = \vec{v}$. For the velocity approximation we use the non-conforming Crouzeix--Raviart element, that was introduced in \cite{CrouzeixRaviart1973}, and is defined by
\begin{equation*}
	\vec{X}_h = \{\vec{v}_h\in \vec{L}^2(\Omega) : \vec{v}_{h}|_T \in \vec{P}_1 \text{ for all } T\in\mathcal{T}_h, \jump{\vec{v}_h}_F(\vec{x}_F) = 0 \text{ for all } F \in \Faces\}.
\end{equation*}
The corresponding pressure approximation uses piecewise constants from the space
\begin{equation*}
	Q_h = \{q_h\in Q : q_{h}|_T \in P_0 \text{ for all } T \in \mathcal{T}_h\},
\end{equation*}
where $P_k$ denotes the space of all polynomials with maximal degree $k$.

Using the space $\Hdiv = \{ \vec{v} \in \vec{L}^2(\Omega) : \nabla\cdot \vec{v} \in L^2(\Omega) \}$ we define the Brezzi--Douglas--Marini and Raviart--Thomas functions of lowest order by
\begin{align*}
	\BDM &= \{ \vec{v}_h \in \Hdiv: \vec{v}_h|_T \in \vec{P}_1\, \forall T\in\mathcal{T}_h, \jump{\vec{v}_h\cdot \vec{n}_F}_F = 0 \, \forall F\in\Faces\}, \\
	\RT &= \{ \vec{v}_h\in \Hdiv: \vec{v}_{h}|_T = \vec{a}_T + b_T(\vec{x}-\vec{x}_T) \, \forall T\in \mathcal{T}_h, \vec{a}_T\in\R^d, b_T \in \R, \\
	&\hphantom{= \{} \jump{\vec{v}_h \cdot \vec{n}_F} = 0\, \forall F\in\Faces \}.
\end{align*}
The Raviart--Thomas function space $\RT$ contains those Brezzi--Douglas--Marini functions from $\BDM$, which have constant normal components on all faces. These normal components define the Raviart--Thomas functions uniquely.

The Crouzeix--Raviart element is not $\vec{H}^1(\Omega)$ conforming, so the standard definitions of the gradient $\nabla$ and divergence $\nabla\cdot$ operators do not make sense for functions in $\vec{X}_h$. Instead we use the notions of the broken gradient $\nabla_h:\vec{X}\oplus\vec{X}_h\rightarrow L^2(\Omega)^{d\times d}$ and broken divergence $\nabla_h\cdot(\cdot) : \vec{X}\oplus\vec{X}_h\rightarrow L^2(\Omega)$, which define the derivatives elementwise for all $T\in\mathcal{T}_h$ by
\begin{equation*}
	(\nabla_h\vec{v}_h)|_T = \nabla(\vec{v}_h|_T),	\quad \text{and} \quad (\nabla_h\cdot \vec{v}_h)|_T = \nabla\cdot(\vec{v}_h|_T),
\end{equation*}
and which are on $\vec{X}$ equivalent to the standard operators, see e.g. \cite[Sections 1.2.5, 1.2.6]{DiPietroErn2012}. 
The discrete gradient norm for the space $\vec{X}\oplus\vec{X}_h$ is defined by
\begin{equation*}
	\norm{\vec{v}_h}_{1,h} = \left(\int_\Omega \nabla_h\vec{v}_h : \nabla_h\vec{v}_h\right)^{\nicefrac{1}{2}} = \norm{\nabla_h \vec{v}_h}_0.
\end{equation*}

We define the three interpolation operators $\cri:\vec{X} \rightarrow \vec{X}_h$, $\rti:\vec{X}\oplus\vec{X}_h\rightarrow \RT$ and $\bdmih:\vec{X}\oplus\vec{X}_h\rightarrow \BDM$ for the Crouzeix--Raviart, Raviart--Thomas and Brezzi--Douglas--Marini interpolation by
\begin{align*}
	\cri \vec{v} (\vec{x}_F) &= 	\frac{1}{\abs{F}}\int_F \vec{v}, \text{ for all } F\in \Faces, \\
	\vec{n}_F \cdot \rti \vec{v} (\vec{x}_F) &= \frac{1}{\abs{F}}\int_F \vec{v} \cdot \vec{n}_F , \text{ for all } F \in \Faces, \\
	\int_F (\bdmih\vec{v})\cdot \vec{n}_F p_h &= 	\begin{cases}
													\int_F \avg{\vec{v}\cdot \vec{n}_F} p_h, &\text{for all } F \in \innerFaces, \\
													\int_F (\rti \vec{v})\cdot \vec{n}_F p_h, &\text{for all } F \in \boundaryFaces,
												\end{cases} \quad \text{for all } p_h \in P_1(F).
\end{align*}
This definition of $\bdmih$ on the boundary facets is necessary in order for $\bdmih \vec{v}_h \cdot \vec{n}$ to vanish along the boundary for $\vec{v}_h \in \vec{X}_h$, and thus to establish the $L^2$ orthogonality with gradients.
Note that due to continuity at the facet barycenters $\vec{x}_F$ and the use of the average, all interpolation operators are well-defined for all elements of $\vec{X}\oplus \vec{X}_h$.
Additionally let $\pi_h:Q\rightarrow Q_h$ be the $L^2$ projection onto the discrete pressure space, which is defined for $p\in Q$ by
\begin{equation*}
	(\pi_h p,q_h) = (p,q_h),\quad \text{for all } q_h\in Q_h.
\end{equation*}

The Raviart--Thomas interpolation and the $L^2$ projection operators will be applied to matrices and vectors, respectively, which we then denote by $\Rti$ and and $\Pi_h$. The operators then have to be understood as acting row by row.

Using the discrete bilinear and linear forms
\begin{alignat*}{2}
	&a_h:\vec{X}_h\times\vec{X}_h \rightarrow \R,&&\quad a_h(\vec{u}_h,\vec{v}_h) = \nu \int_\Omega \nabla_h\vec{u}_h : \nabla_h\vec{v}_h,	\\
	&b_h:\vec{X}_h\times Q_h \rightarrow \R,&&\quad b_h(\vec{v}_h, q_h) = -\int_\Omega q_h \nabla_h\cdot\vec{v}_h,\\
	&l_h:\vec{X}_h\rightarrow \R,&&\quad l_h(\vec{v}_h) = \int_\Omega \vec{f} \cdot \hdivi \vec{v}_h , %
\end{alignat*}
with $\hdivi \in \{\rti, \bdmih\}$, see \cite{Linke2014,BrenneckeLinkeMerdonSchoberl2015}, we get a discrete weak formulation of \eqref{eq:stokescont}: find $(\vec{u}_h,p_h)\in \vec{X}_h\times Q_h$ so that 
\begin{equation}\label{eq:stokesweak}
	\begin{split}
		a_h(\vec{u}_h,\vec{v}_h) + b_h(\vec{v}_h,p_h) &= l_h(\vec{v}_h),	\\
		b_h(\vec{u}_h,q_h) &= 0,%
	\end{split}
\end{equation}
holds for all $(\vec{v}_h,q_h)\in \vec{X}_h\times Q_h$. 
Like in the continuous case, now using the space of discretely divergence constrained functions
\begin{equation*}
	\vec{V}^0_h = \left\{\vec{v}_h\in \vec{X}_h : b(\vec{v}_h, q_h) = 0 \text{ for all } q_h\in Q_h\right\},
\end{equation*}
we can write this problem in the elliptic form \cite{LinkeMerdonWollner2017,BrenneckeLinkeMerdonSchoberl2015,GiraultRaviart1986}: find $\vec{u}_h\in \vec{V}^0_h$ so that
\begin{equation}\label{eq:stokeselliptic}
	a_h(\vec{u}_h,\vec{v}_h) = l_h(\vec{v}_h), \quad \text{for all }\vec{v}_h\in\vec{V}^0_h.	
\end{equation}

The reason for the particular choice of the linear form $l_h$ is described in detail for the Raviart--Thomas interpolation in \cite{Linke2014} and subsequently for various other cases in \cite{LinkeMerdonWollner2017,LinkeMatthiesTobiska2016,JohnLinkeMerdonNeilanRebholz2017,LinkeMerdon2016,LedererLinkeMerdonSchoberl2017}.
The fundamental idea is, that by using an interpolation operator that maps discretely divergence free functions to exactly divergence free functions, which are $L^2$ orthogonal to irrotational functions, it is possible to achieve pressure-robustness for the discrete formulation.

Concluding this section, we state a commutative property for the introduced interpolation operators.
\begin{lemma}\label{lem:commutative}
	For all $\vec{v}\in \vec{X}$ there holds
	\begin{align*}
		\nabla_h \cdot \cri\vec{v} = \pi_h(\nabla\cdot\vec{v}), \\
		\nabla \cdot \rti \vec{v} = \pi_h(\nabla\cdot\vec{v}), \\
		\nabla \cdot \bdmih \vec{v} = \pi_h(\nabla\cdot\vec{v}).
	\end{align*}
\end{lemma}
\begin{proof}
The properties follow by the divergence theorem and the definition of the interpolation operators. See also \cite{CrouzeixRaviart1973,Nedelec1980,Nedelec1986}.
\end{proof}
In particular this means $\nabla_h \cdot \cri\vec{v} = \nabla \cdot \rti \vec{v} = \nabla \cdot \bdmih \vec{v} = 0$ for $\vec{v} \in \vec{V}^0$.

\section{Some properties of the Crouzeix--Raviart element concerning anisotropic triangulations}\label{sec:CR_results}
The Crouzeix--Raviart element has some properties, which make it very suitable for settings with anisotropic triangulations. In this section we collect some, mostly known, results as an overview of the anisotropic properties of the element.

By \cite[Lemma II.1.1]{GiraultRaviart1986}, a discrete Fortin operator $\fori : \vec{X} \rightarrow \vec{X}_h$ is defined by the properties
\begin{equation}\label{eq:fortinoperator}
	(\nabla_h\cdot\vec{v},q_h)_{L^2(\Omega)} = (\nabla_h\cdot\fori \vec{v},q_h)_{L^2(\Omega)} \quad \text{for all } q_h \in Q_h,
\end{equation}
and
\begin{equation*}
	\exists C_F>0: \norm{\fori \vec{v}}_{1,h} \leq C_F \norm{\vec{v}}_{1,h},
\end{equation*}
with $C_F$ independent of $h$. The existence of such an operator is equivalent to the discrete inf-sup condition holding with a constant $\widetilde{\beta} >0$, independent of $h$.

\begin{lemma}\label{lem:CRFortin}
The Crouzeix--Raviart interpolator $\cri$ is a Fortin operator on arbitrary meshes with Fortin constant $C^{\text{CR}}_F = 1$, i.e. the estimate
	\begin{equation} \label{eq:CRFortin}
		\norm{\cri \vec{v}}_{1,h} \leq \norm{\vec{v}}_{1,h}
	\end{equation}
	holds.
\end{lemma}
\begin{proof}
For the proof see \cite[Lemma 2]{ApelNicaiseSchoberl2001:2} and the comment after \cite[Corollary 1]{ApelNicaiseSchoberl2001:2}.
\end{proof}

Using this result, we get the inf-sup condition of the Crouzeix--Raviart element for the Stokes problem on arbitrary meshes with a constant independent of the mesh. Another proof for the inf-sup condition is given in \cite[Theorem 3.151]{John2016}, and see also \cite[Section 3, p. 23]{AcostaDuran1999}.
\begin{lemma}\label{lem:infsup}
	Let $h>0$ and $\vec{v}_h \in \vec{X}_h$, then there is a constant $\widetilde{\beta} > 0$ independent of $h$, so that the estimate 
	\begin{equation*}
		\inf_{q_h\in Q_h} \sup_{\vec{v}_h\in \vec{X}_h} \frac{b_h(\vec{v}_h,q_h)}{\norm{\vec{v}_h}_{1,h} \norm{q_h}_0} \geq \widetilde{\beta}
	\end{equation*}
	holds for arbitrary meshes.
\end{lemma}
\begin{proof}
By \eqref{eq:fortinoperator} and \eqref{eq:CRFortin}, we have the estimate
\begin{equation*}
	\sup_{\vec{v}_h\in \vec{X}_h} \frac{b_h(\vec{v}_h,q_h)}{\norm{\vec{v}_h}_{1,h}} \geq \sup_{\vec{v}\in \vec{X}} \frac{b_h(\cri\vec{v},q_h)}{\norm{\cri\vec{v}}_{1,h}} = \sup_{\vec{v}\in \vec{X}} \frac{b(\vec{v},q_h)}{\norm{\cri\vec{v}}_{1,h}} \geq \sup_{\vec{v}\in \vec{X}} \frac{b(\vec{v},q_h)}{\norm{\vec{v}}_{1,h}} \geq \beta \norm{q_h}_0,
\end{equation*}
for all $q_h\in Q_h$, where $\beta$ is the continuous inf-sup constant.
\end{proof}
This lemma implies that the discrete inf-sup constant for the Crouzeix--Raviart element is bounded from below by the continuous inf-sup constant for any triangulation, an we may choose $\widetilde{\beta} = \beta$, see \cite[Theorem 3.151]{John2016}. Additionally, it was shown in \cite[Lemma 5]{Gallistl2019}, that the discrete inf-sup constant decreases monotonously when refining a mesh.

The next lemma shows that Crouzeix--Raviart interpolation is as accurate as the standard nodal Lagrange interpolation.
\begin{lemma}
	Let $\vec{v}\in\vec{X}\cap \vec{H}^2(\Omega)$. Then the estimate
	\begin{equation*}
		\norm{\vec{v}-\cri\vec{v}}_{1,h} \leq 2\norm{\vec{v}-\li \vec{v}}_{1,h}
	\end{equation*}
	holds for arbitrary meshes, where $\li$ is the nodal Lagrange interpolation operator.
\end{lemma}
\begin{proof} 
	The proof follows part of the proof of \cite[Lemma 4.53]{John2016}, but note that no condition on the mesh is required for this section of the proof. 
	Using the triangle inequality, the property $\cri\li\vec{v} = \li\vec{v}$ and \Cref{lem:CRFortin}, we get the desired estimate
	\begin{equation*}
	\begin{split}
		\norm{\vec{v}-\cri\vec{v}}_{1,h} &\leq \norm{\vec{v}-\li\vec{v}}_{1,h} + \norm{\li\vec{v}-\cri\vec{v}}_{1,h} \\
		&= \norm{\vec{v}-\li\vec{v}}_{1,h} + \norm{\cri(\li\vec{v}-\vec{v})}_{1,h} \\
		&\leq 2\norm{\vec{v}-\li \vec{v}}_{1,h}. \qedhere
	\end{split}
	\end{equation*}
\end{proof}

In \cite[Lemma 3]{ApelNicaiseSchoberl2001:2} the following interpolation error estimate for the Crouzeix--Raviart interpolator for triangulations satisfying a maximum angle condition was shown. We state the result without proof.
\begin{lemma}\label{lem:crinterpolation}
	Let $\vec{v}\in\vec{X}\cap \vec{H}^2(\Omega)$ and let the mesh satisfy $\MAC(\bar{\phi})$. Then we have the estimate
	\begin{equation*}
		\norm{\vec{v}-\cri\vec{v}}_{1,h} \leq Ch\abs{\vec{v}}_2.
	\end{equation*}
\end{lemma}

The next lemma states, that the discretely divergence constrained Crouzeix--Raviart functions can be used to approximate the continuously constrained $\vec{H}^1_0(\Omega)$ functions.
\begin{lemma}
	Let $\vec{w}\in\vec{V}^0$. Then the estimate
	\begin{equation*}
		\inf_{\vec{w}_h\in\vec{V}^0_h} \norm{\vec{w}-\vec{w}_h}_{1,h} \leq 2\inf_{\vec{v}_h\in\vec{X}_h}\norm{\vec{w}-\vec{v}_h}_{1,h}
	\end{equation*}
	holds for an arbitrary triangulation.
\end{lemma}
\begin{proof}
	Let $\vec{v}_h\in\vec{X}_h$ be arbitrary and set $\vec{z}_h=\cri (\vec{w}-\vec{v}_h)\in \vec{X}_h$. Then we have $\norm{\vec{z}_h}_{1,h} \leq \norm{\vec{w}-\vec{v}_h}_{1,h}$ and $(\nabla_h\cdot \vec{z}_h,q_h) = (\nabla_h\cdot (\vec{w}-\vec{v}_h),q_h)$ for all $q_h\in Q_h$. We also get $\vec{w}_h = \vec{z}_h + \vec{v}_h \in \vec{V}^0_h$, because
	\begin{multline*}
		(\nabla_h\cdot \vec{w}_h,q_h) = (\nabla_h\cdot \vec{z}_h,q_h) + (\nabla_h \cdot \vec{v}_h, q_h) = (\nabla_h\cdot(\vec{w}-\vec{v}_h),q_h) + (\nabla_h\cdot \vec{v}_h,q_h) \\
		= (\nabla_h\cdot \vec{w},q_h) = 0.
	\end{multline*}
	Now using the triangle inequality we get the statement of the lemma
	\begin{equation*}
		\norm{\vec{w}-\vec{w}_h}_{1,h} \leq \norm{\vec{w}-\vec{v}_h}_{1,h} + \norm{\vec{z}_h}_{1,h} \leq 2\norm{\vec{w}-\vec{v}_h}_{1,h}. \qedhere
	\end{equation*}
\end{proof}

\section{A-priori error analysis}\label{sec:a_priori_errors}
As our method uses an interpolation operator on the velocity test functions in the linear form $l_h$, we need to estimate the additional consistency error of this approach. The proofs are mainly analogous to \cite{Linke2014}.

Before we get to the consistency error, we need error estimates of the Brezzi--Douglas--Marini and Raviart--Thomas interpolation on anisotropic elements, which we get from \cite{AcostaApelDuranLombardi2011,ApelKempf2019}.
Keeping in mind that by assumption the general triangulation $\mathcal{T}_h$ satisfies a maximum angle condition $\MAC(\bar{\phi})$, we have the following estimates, where we take $\hdivi \in \{\bdmih, \rti\}$.
\begin{lemma}\label{lem:rtinterpolation}%
	Let $\vec{v} \in \vec{X} \oplus \vec{X}_h$, then
	\begin{equation*}
		\norm{\vec{v}-\hdivi\vec{v}}_{0} \leq Ch \norm{\vec{v}}_{1,h},
	\end{equation*}
	where the constant $C$ depends only on $\bar{\phi}$.
	
	If an element $T\in\mathcal{T}_h$ additionally satisfies $\RVP(\bar{c})$, where $\vec{p}_{T,d+1}$ denotes the element's regular vertex and $\vec{l}_{T,i}$, $h_{T,i}$, $i\in\{1,\ldots,d\}$, the vectors and lengths from \Cref{def:regular_vertex_property}. Then for $\vec{v} \in \vec{X}\oplus\vec{X}_h$ there is a constant $C$ depending only on $\bar{c}$, so that the estimate
	\begin{equation*}
		\norm{\vec{v}-\hdivi \vec{v}}_{0,T} \leq C \left( h_T \norm{\nabla\cdot\vec{v}}_{0,T} + \sum_{i=1}^d h_{T,i} \norm{\pdv{\vec{v}}{\vec{l}_{T,i}}}_{0,T} \right)
	\end{equation*}
	holds.
\end{lemma}
\begin{proof}
	The proof can be found for the Raviart--Thomas interpolation in \cite{AcostaApelDuranLombardi2011} and for the Brezzi--Douglas--Marini interpolation in \cite{ApelKempf2019}. 
	For functions from $\vec{X}$, the slightly different definitions of the operator $\bdmih$ and the interpolation operator used in \cite{ApelKempf2019} are equivalent. For functions from $\vec{X}_h$, the interpolation error estimates can be extended. A proof for the isotropic case is given in \cite[Lemma 3.3]{LinkeMerdonNeilanNeumann2018}, which can be transfered to our setting.
\end{proof}

The following technical lemma prepares the estimate of the consistency error. The proof is analogous to \cite[Lemma 5]{Linke2014}, where we now use the interpolation error estimates from \Cref{lem:rtinterpolation}.
\begin{lemma}\label{lem:errorestimate1}
	Let $\vec{v}\in \vec{X} \cap \vec{H}^2(\Omega)$ and $\vec{w}\in \vec{X} \oplus \vec{X}_h$, then the estimate
	\begin{equation*}
		\abs{\int_\Omega \left[\nabla_h \vec{v} : \nabla_h \vec{w} + \Delta \vec{v} \cdot \hdivi \vec{w}\right]  } \leq Ch \abs{\vec{v}}_2 \norm{\vec{w}}_{1,h}
	\end{equation*}
	holds. If additionally every element $T \in \mathcal{T}_h$ satisfies $\RVP(\bar{c})$, then using the notation of \Cref{lem:rtinterpolation} we get the estimate
	\begin{multline*}
		\abs{\int_\Omega \left[\nabla_h \vec{v} : \nabla_h \vec{w} + \Delta \vec{v} \cdot \hdivi \vec{w}\right]  } \leq \\
		C \norm{\vec{w}}_{1,h} \left(h\norm{\Delta\vec{v}}_0 + \sum_{T\in\mathcal{T}_h} \sum_{i=1}^d \sum_{j=1}^d h_{T,j} \norm{\pdv{\nabla v_i}{\vec{l}_{T,j}}}_{0,T}\right).
	\end{multline*}
\end{lemma}
\begin{proof}
	Using the triangle inequality we get the estimate
	\begin{multline}\label{eq:errorestimate1_1}
		\abs{\int_\Omega \left[\nabla_h \vec{v} : \nabla_h \vec{w} + \Delta \vec{v} \cdot \hdivi \vec{w}\right]  } \leq \abs{\int_\Omega \left[\nabla_h \vec{v} : \nabla_h \vec{w} + \Delta \vec{v} \cdot \vec{w}\right]  } \\ 
		+ \abs{\int_\Omega \Delta \vec{v} \cdot \left( \hdivi \vec{w} - \vec{w} \right)  }.
	\end{multline}
	The second term can be estimated using the Cauchy-Schwarz inequality and \Cref{lem:rtinterpolation} and we get the result
	\begin{equation*}
	\begin{split}
		\abs{\int_\Omega \Delta \vec{v} \cdot \left( \hdivi \vec{w} - \vec{w} \right)  } &\leq \norm{\Delta \vec{v}}_0 \norm{\hdivi \vec{w} - \vec{w}}_0 \\
		&\leq Ch \norm{\Delta \vec{v}}_0 \norm{\vec{w}}_{1,h} \leq Ch \abs{\vec{v}}_2 \norm{\vec{w}}_{1,h}.
	\end{split}
	\end{equation*}
	Using Green's identity we get a new representation of the first term on the right hand side of \eqref{eq:errorestimate1_1}:
	\begin{equation}\label{eq:errorestimate1_2}
		\abs{\int_\Omega \left[\nabla_h \vec{v} : \nabla_h \vec{w} + \Delta \vec{v} \cdot \vec{w}\right]  } = \abs{\sum_{T\in \mathcal{T}_h} \int_{\partial T} (\nabla\vec{v} \cdot \vec{n}) \cdot \vec{w}  }.
	\end{equation}
	Recall that we use the symbols $\Rti$ and $\Pi_h$ to indicate the row-by-row application of the Raviart--Thomas interpolation and the $L^2$ projection into the discrete pressure space on matrices and vectors, respectively.
	As described in the proof of \cite[Lemma 3.1]{AcostaDuran1999}, we observe that $(\Rti \nabla \vec{v})\cdot \vec{n}$ is constant on all faces and continuous across the interelement boundaries, and $\vec{v}$ vanishes at the boundary, so that we get 
	\begin{equation*}
		\sum_{T\in \mathcal{T}_h} \int_{\partial T} (\Rti\nabla\vec{v} \cdot \vec{n}) \cdot \vec{w}    = 0.
	\end{equation*}
	for all $\vec{w}\in \vec{X}\oplus \vec{X}_h$. Thus we can subtract this term from the right hand side of \eqref{eq:errorestimate1_2}, and using the divergence theorem we get
	\begin{align}
		&\left|\sum_{T\in \mathcal{T}_h} \int_{\partial T} (\nabla\vec{v} \cdot \vec{n}) \cdot \vec{w}    \right| = \abs{\sum_{T\in \mathcal{T}_h} \int_{\partial T} ((\nabla\vec{v} - \Rti \nabla\vec{v}) \cdot \vec{n}) \cdot \vec{w}   }	\nonumber\\
		&\quad= \abs{\sum_{T\in \mathcal{T}_h} \int_{T} \nabla \cdot ((\nabla\vec{v} - \Rti \nabla\vec{v}) \cdot \vec{w})   } \nonumber\\
		&\quad= \abs{\sum_{T\in \mathcal{T}_h} \int_{T}\left[ (\nabla \cdot (\nabla\vec{v} - \Rti \nabla\vec{v})) \cdot \vec{w} + (\nabla\vec{v} - \Rti\nabla\vec{v}):\nabla\vec{w}\right]  } \nonumber\\
		&\quad\leq \abs{\sum_{T\in \mathcal{T}_h} \int_{T}(\nabla \cdot (\nabla\vec{v} - \Rti \nabla\vec{v})) \cdot \vec{w}  } + \abs{\sum_{T\in \mathcal{T}_h} \int_{T}(\nabla\vec{v} - \Rti\nabla\vec{v}):\nabla\vec{w}   }.\label{eq:errorestimate1_3}
	\end{align}
	For the first term on the right hand side, observe that due to \Cref{lem:commutative} we have
	\begin{equation*}
		\nabla\cdot\Rti\nabla\vec{v}=\Pi_h(\nabla\cdot\nabla\vec{v}) = \Pi_h\Delta\vec{v},
	\end{equation*}
	thus using $L^2$ orthogonality we can estimate
	\begin{multline*}
		\abs{\sum_{T\in \mathcal{T}_h} \int_{T}(\nabla \cdot (\nabla\vec{v} - \Rti \nabla\vec{v})) \cdot \vec{w}  } = \abs{\sum_{T\in \mathcal{T}_h} \int_{T}(\Delta\vec{v} - \Pi_h \Delta\vec{v}) \cdot \vec{w}  } \\
		= \abs{\sum_{T\in \mathcal{T}_h} \int_{T}\Delta\vec{v} \cdot (\vec{w} - \Pi_h\vec{w})  } \leq Ch \norm{\Delta \vec{v}}_0\norm{\vec{w}}_{1,h} \leq Ch  \abs{\vec{v}}_2\norm{\vec{w}}_{1,h}.
	\end{multline*}
	Using the Cauchy-Schwarz inequality, the interpolation estimates from \Cref{lem:rtinterpolation} and some basic calculations, the second term on the right hand side of \eqref{eq:errorestimate1_3} can be estimated by
	\begin{align*}
		&\hspace{-1cm}\abs{\sum_{T\in \mathcal{T}_h} \int_{T}(\nabla\vec{v} - \Rti\nabla\vec{v}):\nabla\vec{w}   } \leq \sum_{T\in\mathcal{T}_h}\norm{\nabla\vec{v}-\Rti\nabla\vec{v}}_{0,T} \norm{\nabla_h\vec{w}}_0  \\
		&\leq \sum_{T\in\mathcal{T}_h}\left[\sum_{i=1}^d \norm{\nabla v_i-\rti\nabla v_i}^2_{0,T}\right]^{\nicefrac{1}{2}} \norm{\vec{w}}_{1,h} \\
		&\leq C \sum_{T\in\mathcal{T}_h}\left[\sum_{i=1}^d \left(h_T \norm{\Delta v_i}_{0,T} + \sum_{j=1}^d h_{T,j}\norm{\pdv{\nabla v_i}{\vec{l}_{T,j}}}_{0,T} \right)^2 \right]^{\nicefrac{1}{2}} \norm{\vec{w}}_{1,h} \\
		&\leq C \norm{\vec{w}}_{1,h} \left(h\norm{\Delta\vec{v}}_0 + \sum_{T\in\mathcal{T}_h} \sum_{i=1}^d \sum_{j=1}^d h_{T,j} \norm{\pdv{\nabla v_i}{\vec{l}_{T,j}}}_{0,T}\right) \leq C h \abs{\vec{v}}_2 \norm{\vec{w}}_{1,h}.
	\end{align*}
	Combining the individual estimates we get the statement of the lemma.
\end{proof}

\begin{lemma}\label{lem:consistencyerror}
	Let $(\vec{u},p)\in \vec{H}^2(\Omega) \times H^1(\Omega)$ hold for the solution $(\vec{u}, p)$ of \eqref{eq:stokescont}. Then the estimate
	\begin{equation*}
		\frac{1}{\nu}\sup_{\vec{w}\in\vec{V}^0\oplus \vec{V}^0_h} \frac{\abs{a_h(\vec{u},\vec{w}) - l_h(\vec{w})}}{\norm{\vec{w}}_{1,h}} \leq Ch \abs{\vec{u}}_2
	\end{equation*}
	holds. If additionally every element $T \in \mathcal{T}_h$ satisfies $\RVP(\bar{c})$, then using the notation of \Cref{lem:rtinterpolation} we have the estimate
	\begin{equation*}
		\frac{1}{\nu}\sup_{\vec{w}\in\vec{V}^0\oplus \vec{V}^0_h} \frac{\abs{a_h(\vec{u},\vec{w}) - l_h(\vec{w})}}{\norm{\vec{w}}_{1,h}} \leq C \left(h\norm{\Delta\vec{u}}_0 + \sum_{T\in\mathcal{T}_h} \sum_{i=1}^d \sum_{j=1}^d h_{T,j} \norm{\pdv{\nabla u_i}{\vec{l}_{T,j}}}_{0,T}\right).
	\end{equation*}
\end{lemma}
\begin{proof}
	Let $0\neq \vec{w} \in \vec{V}^0\oplus\vec{V}^0_h$. Using partial integration yields
	\begin{align*}
		(\nabla p, \hdivi \vec{w}) = -(p, \nabla\cdot \hdivi \vec{w}) + (p,\hdivi \vec{w}\cdot \vec{n})_{\partial \Omega} = 0,
	\end{align*}	
	due to the choice of $\vec{w}$ and the boundary conditions in the spaces $\BDM$ and $\RT$.
	With this equality we get
	\begin{align*}
		\frac{1}{\nu} \abs{a_h(\vec{u},\vec{w}) - l_h(\vec{w})} &= \frac{1}{\nu} \abs{\int_\Omega \left[\nu\nabla_h\vec{u} : \nabla_h \vec{w} - \vec{f} \cdot \hdivi\vec{w}\right]   }	\\
		&= \frac{1}{\nu} \abs{\int_\Omega \left[\nu\nabla_h\vec{u} : \nabla_h \vec{w} +(\nu\Delta\vec{u}-\nabla p) \cdot \hdivi\vec{w}\right]   } \\
		&= \abs{\int_\Omega \left[\nabla_h\vec{u} : \nabla_h \vec{w} + \Delta\vec{u} \cdot \hdivi\vec{w}\right]   }.
	\end{align*}
	Now using the two results from \Cref{lem:errorestimate1} yields the statement of the lemma.
\end{proof}

\begin{theorem}
	Let $(\vec{u},p)\in \vec{H}^2(\Omega) \times H^1(\Omega)$ hold for the solution $(\vec{u}, p)$ of \eqref{eq:stokescont}, and let $(\vec{u}_h, p_h)$ be the discrete solution of \eqref{eq:stokesweak}. Then the estimates
	\begin{align}
		&\norm{\vec{u}-\vec{u}_h}_{1,h} \leq 2 \inf_{\vec{v}_h\in\vec{V}^0_h} \norm{\vec{u} - \vec{v}_h}_{1,h} + Ch\abs{\vec{u}}_2,	\label{eq:errorestimateu}\\
		&\norm{\pi_h p - p_h}_0 \leq \frac{\nu}{\widetilde{\beta}}\left( 2 \inf_{\vec{v}_h\in\vec{V}^0_h} \norm{\vec{u} - \vec{v}_h}_{1,h} + C h \abs{\vec{u}}_2\right), \label{eq:errorestimatepbestapprox}\\
		&\norm{p-p_h}_0 \leq \inf_{q_h\in Q_h} \norm{p-q_h}_0 + \frac{\nu}{\widetilde{\beta}}\left( 2 \inf_{\vec{v}_h\in\vec{V}^0_h} \norm{\vec{u} - \vec{v}_h}_{1,h} + C h \abs{\vec{u}}_2\right), \label{eq:errorestimatep}
	\end{align}
	hold. If additionally every element $T \in \mathcal{T}_h$ satisfies $\RVP(\bar{c})$, then using the notation of \Cref{lem:rtinterpolation} we get the estimates
	\begin{align}
		&\norm{\vec{u}-\vec{u}_h}_{1,h} \leq 2 \inf_{\vec{v}_h\in\vec{V}^0_h} \norm{\vec{u} - \vec{v}_h}_{1,h} \nonumber\\ 
		&\hspace{5cm} + C \left(h\norm{\Delta\vec{u}}_0 + \sum_{T\in\mathcal{T}_h} \sum_{i=1}^d \sum_{j=1}^d h_{T,j} \norm{\pdv{\nabla u_i}{\vec{l}_{T,j}}}_{0,T}\right),	\label{eq:errorestimateu2}\\
		&\norm{\pi_h p - p_h}_0 \leq \frac{\nu}{\widetilde{\beta}}\left[ 2 \inf_{\vec{v}_h\in\vec{V}^0_h} \norm{\vec{u} - \vec{v}_h}_{1,h}  \right.\nonumber \\
		&\hspace{5cm} +\left.  C \left(h\norm{\Delta\vec{u}}_0 + \sum_{T\in\mathcal{T}_h} \sum_{i=1}^d \sum_{j=1}^d h_{T,j} \norm{\pdv{\nabla u_i}{\vec{l}_{T,j}}}_{0,T}\right) \right], \label{eq:errorestimatepbestapprox2}\\
		&\norm{p-p_h}_0 \leq \inf_{q_h\in Q_h} \norm{p-q_h}_0 + \frac{\nu}{\widetilde{\beta}}\left[ 2 \inf_{\vec{v}_h\in\vec{V}^0_h} \norm{\vec{u} - \vec{v}_h}_{1,h}  \right.\nonumber \\
		&\hspace{5cm} +\left.  C \left(h\norm{\Delta\vec{u}}_0 + \sum_{T\in\mathcal{T}_h} \sum_{i=1}^d \sum_{j=1}^d h_{T,j} \norm{\pdv{\nabla u_i}{\vec{l}_{T,j}}}_{0,T}\right) \right]. \label{eq:errorestimatep2}
	\end{align}
\end{theorem}
\begin{proof}
	Let $\vec{w}_h = \vec{u}_h - \vec{v}_h\in\vec{V}^0_h$ for arbitrary $\vec{v}_h\in\vec{V}^0_h$, then using \eqref{eq:stokeselliptic} we get
	\begin{align*}
		\nu \norm{\vec{w}_h}^2_{1,h} &= a_h(\vec{w}_h,\vec{w}_h) = a_h(\vec{u}_h - \vec{v}_h,\vec{w}_h) \\
		&= a_h(\vec{u} - \vec{v}_h,\vec{w}_h) + a_h(\vec{u}_h,\vec{w}_h) - a_h(\vec{u},\vec{w}_h) \\
		&= a_h(\vec{u} - \vec{v}_h,\vec{w}_h) + l_h(\vec{w}_h) - a_h(\vec{u},\vec{w}_h) \\
		&\leq \nu\norm{\vec{u}-\vec{v}_h}_{1,h} \norm{\vec{w}_h}_{1,h} + \abs{a_h(\vec{u},\vec{w}_h) - l_h(\vec{w}_h)}.
	\end{align*}		
	Using the triangle inequality and the last inequality we get Strang's second lemma in the form
	\begin{align}\label{eq:strang2}
		\norm{\vec{u}-\vec{u}_h}_{1,h} &= \norm{\vec{u}-\vec{v}_h-\vec{w}_h}_{1,h} \nonumber\\
		&\leq 2\inf_{\vec{v}_h\in\vec{V}^0_h} \norm{\vec{u} - \vec{v}_h}_{1,h} + \frac{1}{\nu} \sup_{\vec{w}_h\in\vec{V}^0_h} \frac{\abs{a_h(\vec{u},\vec{w}_h) - l_h(\vec{w}_h)}}{\norm{\vec{w}_h}_{1,h}}.
	\end{align}
	Applying the bounds for the consistency error from \Cref{lem:consistencyerror} we get \eqref{eq:errorestimateu} and \eqref{eq:errorestimateu2}.
	
	Choosing $q_h = \pi_h p - p_h$ in the discrete inf-sup stability inequality from \Cref{lem:infsup} we get the estimate
	\begin{equation}\label{eq:pressure2}
		\norm{\pi_h p - p_h}_0 \leq \frac{1}{\widetilde{\beta}} \sup_{\vec{v}_h \in \vec{X}_h} \frac{b_h(\vec{v}_h,\pi_h p - p_h)}{\norm{\vec{v}_h}_{1,h}}.
	\end{equation}		
	For the numerator we get
	\begin{equation}\label{eq:pressure3}
		b_h(\vec{v}_h,\pi_h p - p_h) = b_h(\vec{v}_h,\pi_h p - p) + b_h(\vec{v}_h, p - p_h) = b_h(\vec{v}_h, p - p_h),
	\end{equation}
	where the last equality is satisfied since by \Cref{lem:commutative} 
	\begin{equation}\label{eq:pressure3.5}
		\nabla_h \cdot \vec{v}_h = \nabla_h \cdot \cri \vec{v}_h = \pi_h (\nabla \cdot \vec{v}_h) \in Q_h
	\end{equation}
	holds and $\pi_h p - p\in Q_h^\perp$. Again using \Cref{lem:commutative} and \eqref{eq:pressure3.5} we get
	\begin{align*}
		\int_\Omega \left[ -p\nabla_h\cdot \vec{v}_h - \nabla p \cdot \hdivi \vec{v}_h \right]  &= \int_\Omega \left[ -p\nabla_h\cdot \vec{v}_h +  p \nabla \cdot \hdivi \vec{v}_h \right]  	\\
		&= \int_\Omega \left[ -p\nabla_h\cdot \vec{v}_h +  p \pi_h(\nabla \cdot \vec{v}_h) \right]\\
		&= \int_\Omega \left[ -p\nabla_h\cdot \vec{v}_h +  p \nabla_h \cdot \vec{v}_h \right]   = 0,
	\end{align*}
	which we can use to simplify further and estimate
	\begin{align} 
		b_h(\vec{v}_h, p - p_h) &= b_h(\vec{v}_h,p) + a_h(\vec{u}_h,\vec{v}_h) - l_h(\vec{v}_h) \nonumber\\
		&= a_h(\vec{u}_h - \vec{u},\vec{v}_h) + \int_\Omega \left[ \nu \nabla_h \vec{u} : \nabla_h \vec{v}_h - p \nabla_h\cdot \vec{v}_h - \vec{f}\cdot \hdivi \vec{v}_h \right]   \nonumber\\
		&\leq \nu \norm{\vec{u}-\vec{u}_h}_{1,h} \norm{\vec{v}_h}_{1,h} + \nu \int_\Omega \left[ \nabla_h \vec{u} : \nabla_h \vec{v}_h + \Delta \vec{u} \cdot \hdivi \vec{v}_h \right]. \label{eq:pressure4}
	\end{align}
	Combining \eqref{eq:pressure2}, \eqref{eq:pressure3}, \eqref{eq:pressure4} we get
	\begin{equation*}
		\norm{\pi_h p - p_h}_0 \leq \frac{\nu}{\widetilde{\beta}} 		\left( \norm{\vec{u}-\vec{u}_h}_{1,h}  + \sup_{\vec{v}_h \in \vec{X}_h} \frac{\int_\Omega \left[ \nabla_h \vec{u} : \nabla_h \vec{v}_h + \Delta \vec{u} \cdot \hdivi \vec{v}_h \right]}{\norm{\vec{v}_h}_{1,h}}\right).
	\end{equation*}
	Now using \eqref{eq:errorestimateu} or \eqref{eq:errorestimateu2}, and the corresponding estimate from \Cref{lem:errorestimate1}, we get estimates \eqref{eq:errorestimatepbestapprox} and \eqref{eq:errorestimatepbestapprox2}, respectively.

	The remaining estimates \eqref{eq:errorestimatep} and \eqref{eq:errorestimatep2} follow by the triangle inequality and the observation that the $L^2$ projection is the best approximation of $p$ in $Q_h$, i.e.
	\begin{equation*}
		 \norm{p-\pi_h p}_0 = \inf_{q_h\in Q_h} \norm{p-q_h}_0. \qedhere
	\end{equation*}
\end{proof}

For a convex domain and $\hdivi = \bdmih$ we can easily get an optimal $L^2$ error estimate by some standard arguments, using another interpolation error estimate from \cite{ApelKempf2019}, which we state without proof.
\begin{lemma}
	Let $\vec{v} \in \vec{H}^2(\Omega)\cap \vec{X}$ and let $\mathcal{T}_h$ satisfy a maximum angle condition, then the estimate
	\begin{equation*}
		\norm{\vec
		{v} - \bdmih \vec{v}}_{0} \leq C h^2 \abs{\vec{v}}_{2},
	\end{equation*}
	holds.
\end{lemma}

\begin{theorem}\label{th:bdm_l2_estimate}
	Let $\Omega$ be convex, $\mathcal{T}_h$ satisfy a maximum angle condition, $\hdivi = \bdmih$, $(\vec{u},p)\in \vec{H}^2(\Omega) \times H^1(\Omega)$ the solution of \eqref{eq:stokescont} and $(\vec{u}_h,p_h)$ the solution of \eqref{eq:stokesweak}. Then the estimate 
	\begin{equation*}
		\norm{\vec{u}-\vec{u}_h}_{0} \leq C h^2 \abs{\vec{u}}_2
	\end{equation*}
	holds.
\end{theorem}
\begin{proof}
	The proof is entirely analogous to the proof in \cite[Section 4]{BrenneckeLinkeMerdonSchoberl2015}, now using the above interpolation error estimate.
\end{proof}

\begin{remark}
	The proof of an estimate as in \Cref{th:bdm_l2_estimate} for $\hdivi = \rti$ is not possible using this approach, because of the weaker interpolation properties of the operator. Due to observations of the $L^2$ error in the numerical experiments in \Cref{sec:numerical_examples}, we conjecture that such an estimate, which was proven in \cite{LinkeMerdonWollner2017} for the shape regular case, holds true also for the anisotropic case.
\end{remark}

\begin{remark}
	For structured meshes as the ones pictured in \Cref{fig:uexact_mesh_2d}, the estimates \eqref{eq:errorestimateu2} -- \eqref{eq:errorestimatep2} simplify to a certain degree, as $\vec{l}_{T,j} = \pm \vec{e}_j$, where $\vec{e}_j$ are the Cartesian unit vectors, and the norms of the directional derivatives can be written as the regular partial derivatives, 
	\begin{equation*}
		\norm{\pdv{\vec{v}}{\vec{l}_{T,i}}}_{0,T} = \norm{\pdv{\vec{v}}{x_i}}_{0,T}. 
	\end{equation*}
	For a uniform structured mesh, the estimate further simplifies due to $h_{T,j} = h_j$ for all $T\in \mathcal{T}_h$, so that we can write e.g. for \eqref{eq:errorestimateu2}
	\begin{equation*}
		\norm{\vec{u}-\vec{u}_h}_{1,h} \leq 2 \inf_{\vec{v}_h\in\vec{V}^0_h} \norm{\vec{u} - \vec{v}_h}_{1,h} + C \left(h\norm{\Delta\vec{u}}_0 + \sum_{i=1}^d \sum_{j=1}^d h_{j} \norm{\pdv{\nabla u_i}{x_j}}_{0}\right).
	\end{equation*}
\end{remark}

\section{Numerical results}\label{sec:numerical_examples}
We now numerically examine the convergence of the modified Crouzeix--Raviart method with special attention to the behavior in anisotropic settings. 

\subsection{2D example}
We choose an exact solution $(\vec{u},p)$ of the Stokes system on the unit square $\Omega = (0,1)^2$, which is given by
\begin{align*}
	&\vec{u}(\vec{x}) = 	\left(\pdv{\xi}{x_2}, -\pdv{\xi}{x_1}\right), &&p(\vec{x}) = \exp(-\frac{x_1}{\epsilon}) - C(\epsilon),
\end{align*}
where the stream function is defined as $\xi(\vec{x}) = x_1^2(1-x_1)^2x_2^2(1-x_2)^2\exp(-\frac{x_1}{\epsilon})$, and $C(\epsilon)$ is a constant necessary to get vanishing mean pressure. For these functions it holds $(\vec{u}, p) \in \vec{H}^2(\Omega) \times L^2_0(\Omega)$ and $\Delta \vec{u} \in \vec{L}^2(\Omega)$, as required for our theoretical results. 

\begin{figure}[t]
	\centering
	\includegraphics[scale=1]{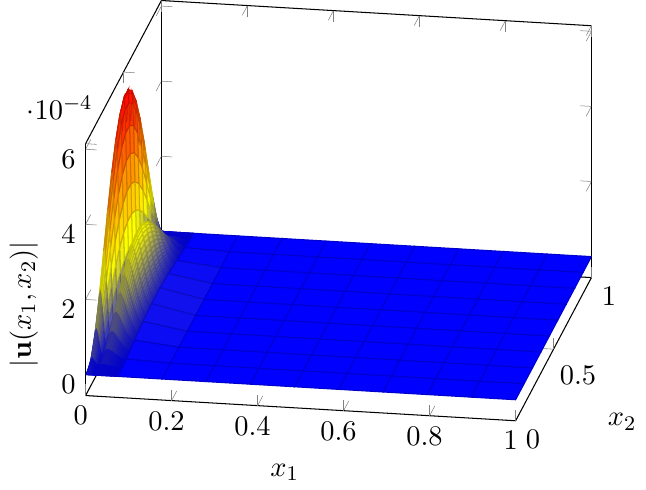}
	\hspace{\fill}
	\includegraphics[scale=1]{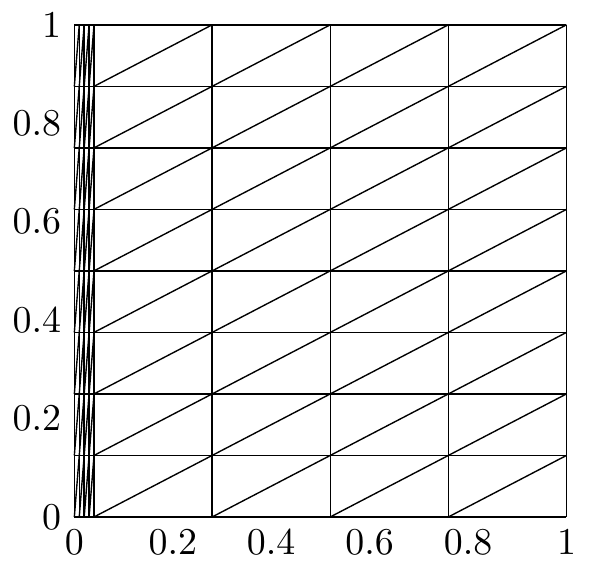}
	\caption{Exact velocity and example mesh for $\epsilon=0.01$, $N=2^3$ used in the calculations}
	\label{fig:uexact_mesh_2d}
\end{figure}
\Cref{fig:uexact_mesh_2d} shows a plot of the magnitude of the exact velocity for the parameter value $\epsilon = 0.01$, where the exponential boundary layer near $x_1=0$ is clearly visible. The layer has a width of $\mathcal{O}(\epsilon)$ and is also present in the pressure solution. The used meshes are of Shishkin type, see the example in \Cref{fig:uexact_mesh_2d}. For a parameter $N\geq 2$ they are constructed by choosing a transition point parameter $\tau \in (0,1)$ and generating a grid of points $(x_1^i, x_2^j)$, 
\begin{align*}
	x_1^i &= \begin{cases}
				i\frac{2\tau}{N}, & 0\leq i\leq \frac{N}{2}, i\in \N, \\
				\tau + \left( i - \frac{N}{2}\right)\frac{2(1-\tau)}{N}, & \frac{N}{2} < i \leq N, i\in \N,
			\end{cases} \\
	x_2^j &= \frac{j}{N}, \quad 0\leq j\leq N, j \in \N.
\end{align*}
Connecting each point to the nearest other grid points by edges, we get a rectangular mesh, then subdividing each rectangle into two triangles leaves us with the desired triangular mesh. 
By this scheme we get a triangulation of $\Omega$ with $n=2 N^2$ elements and an aspect ratio of $\sigma = \frac{\sqrt{1+4\tau^2}}{1+2\tau-\sqrt{1+4\tau^2}}$, see \Cref{fig:uexact_mesh_2d}. The transition point parameter is chosen as $\tau = \min\{\frac{1}{2}, 3\epsilon \abs{\ln(\epsilon)}\}$, which means that approximately three times the boundary layer width are covered by the anisotropic elements. 

\begin{figure}[t]
	\centering
	\includegraphics[scale=1]{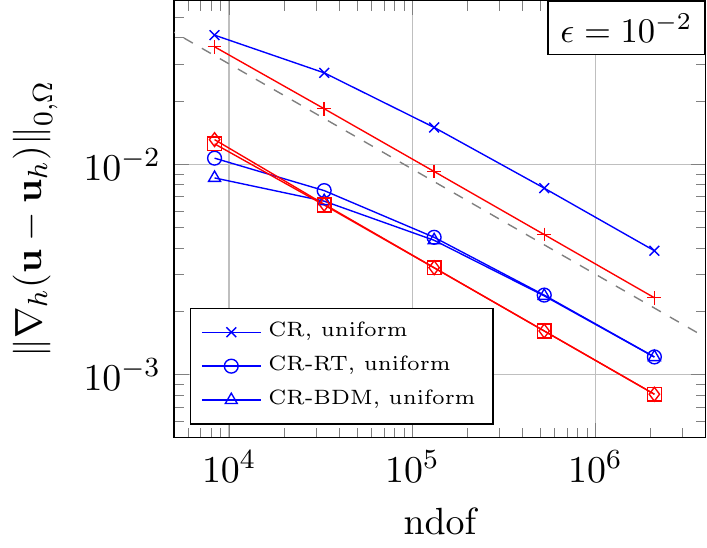}
		\hspace{\fill}
		\includegraphics[scale=1]{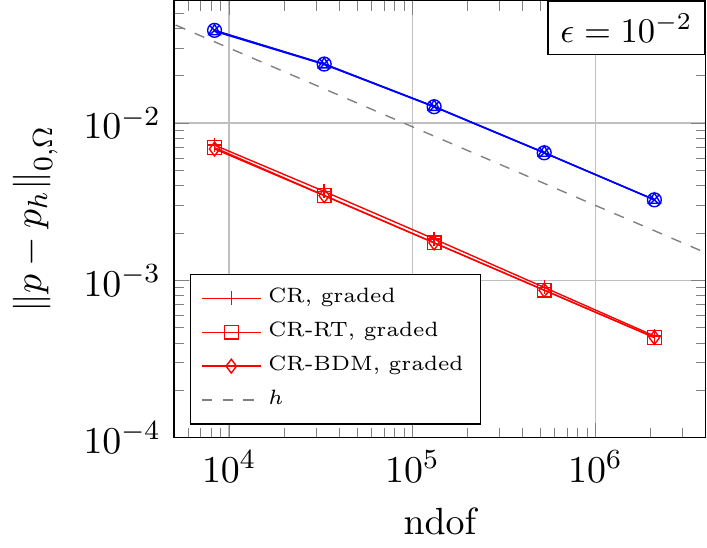}
		\vspace{1em}
		\includegraphics[scale=1]{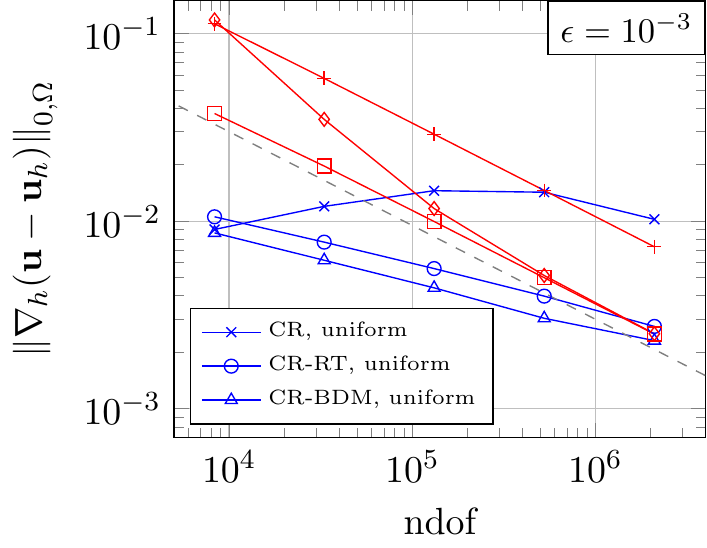}
		\hspace{\fill}
		\includegraphics[scale=1]{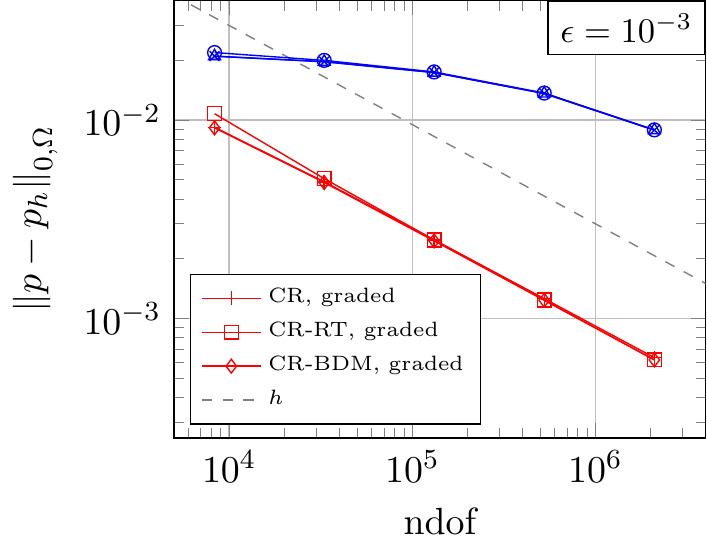}
	\caption{Convergence plots of the discrete velocity and pressure solutions for two values of the parameter $\epsilon$, the dashed lines are the same in all plots}
	\label{fig:results_CNK}
\end{figure}
We performed calculations with parameter values $\nu = 1$, $\epsilon \in \{10^{-2},10^{-3}\}$, both with graded and uniform meshes and the standard Crouzeix--Raviart and the modified method from this paper. In the results shown in \Cref{fig:results_CNK}, two numerical effects are visible. 

The first is due to the anisotropic mesh grading and occurs for both the standard and modified Crouzeix--Raviart methods. Initially when using uniform meshes, the velocity error shows suboptimal convergence rates, until the elements properly resolve the boundary layer, when we observe the theoretical rate of convergence. For anisotropically graded meshes the optimal convergence rate manifests immediately. Once the optimal rate is reached on both types of meshes, the graded mesh produces a lower absolute error. 

The second effect is a result of the pressure-robustness of the modified methods, which in this example leads to significantly reduced errors. We also see, that both modifications $\rti$ and $\bdmih$ lead to similar results.

\subsection{3D example with a singular edge}\label{subsec:example_3d}
We now get to a more relevant three dimensional example, where the beneficial effect of anisotropic mesh grading becomes obvious. Consider the inhomogeneous Stokes problem, i.e. the first two equations of problem \eqref{eq:stokescont}, with the boundary condition $\vec{u} = \vec{g}$ on $\partial \Omega$, on the domain
\begin{equation*}
	\Omega = \{(r\cos(\phi), r\sin(\phi), z) \in \R^3: 0<r<1, 0<\phi<\omega, 0<z<1\}, 
\end{equation*}
where $\omega = \frac{3\pi}{2}$. 

\begin{figure}[t]
	\centering
	\includegraphics[scale=1]{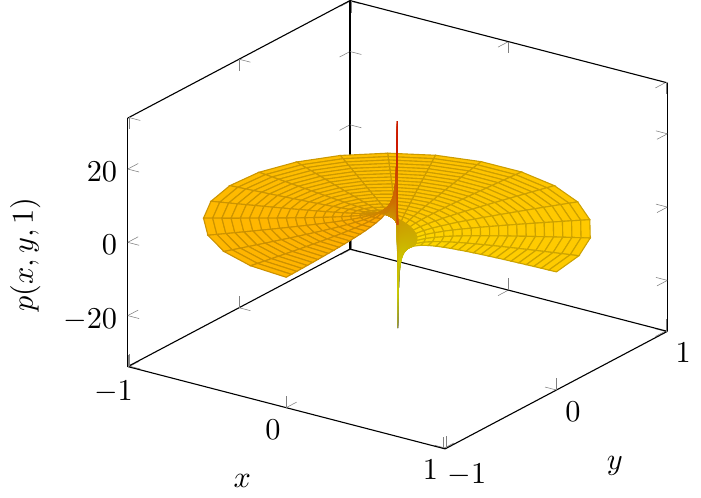}
	\hfill
	\includegraphics[scale=1]{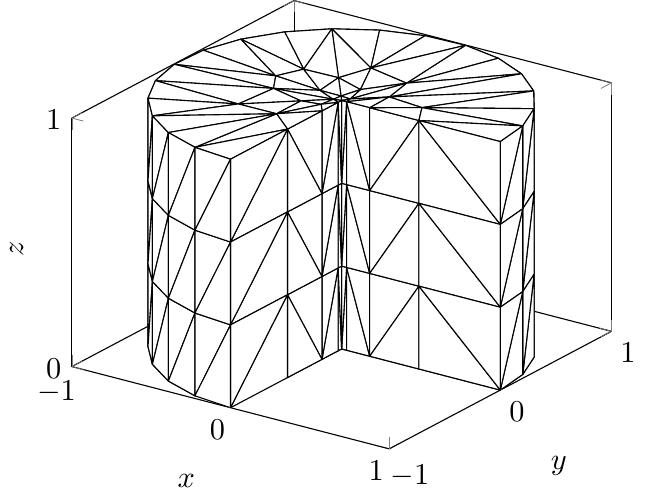}
	\caption{Exact pressure $p(x,y,1)$ with singularity at $z$-axis, anisotropically graded mesh}
	\label{fig:mesh_3d}
\end{figure}
For the convergence calculations we use, as before, the method of manufactured solutions, with exact velocity and exact pressure defined by
\begin{align*}
	\vec{u} &= 	\begin{pmatrix}
					z r^\lambda [-\lambda\sin(\phi)\cos(\lambda(\omega-\phi)+\phi)+\lambda\sin(\omega-\phi)\cos(\lambda\phi-\phi)+\sin(\lambda(\omega-\phi))] \\
					z r^\lambda [\sin(\lambda\phi)-\lambda\sin(\phi)\sin(\lambda(\omega-\phi)+\phi)-\lambda\sin(\omega-\phi)\sin(\lambda\phi-\phi)] \\
					r^{\nicefrac{2}{3}}\sin\left(\frac{2}{3}\phi\right)
				\end{pmatrix}, \\
	p &= 2\lambda z r^{\lambda-1} [\sin((\lambda-1)\phi+\omega)+\sin((\lambda-1)\phi-\lambda\omega)],
\end{align*}
where the parameter $\lambda$ is the smallest positive solution of $\sin(\omega \lambda) = \lambda$, i.e. $\lambda \approx 0.54448$. The singular nature of the exact pressure along the edge at $r=0$ is illustrated in \Cref{fig:mesh_3d}.
The data functions are obtained by $\vec{f} = -\nu\Delta\vec{u}+\nabla p$ and $\vec{g} = \vec{u}|_{\partial \Omega}$. Elementary calculations show that $\nabla \cdot \vec{u} = 0$ and $\int_\Omega p = 0$.

This example was examined in \cite[Section 4]{ApelNicaiseSchoberl2001:2} for the standard Crouzeix--Raviart method, where it illustrated the result that anisotropic mesh grading towards the singular edge leads to an optimal convergence rate, while with uniform meshes the convergence rate in non-convex settings deteriorates because of the low regularity of the solution. 
Due to this low regularity, $(\vec{u},p) \notin \vec{H}^2(\Omega) \times H^1(\Omega)$, $\Delta\vec{u}\in \vec{L}^q(\Omega)$, $1\leq q < \frac{2}{2-\lambda}$, and the assumed inhomogeneous boundary conditions, this example leaves the theoretical framework of our prior analysis. 
Although this is the case and no thorough analysis has been done yet, the numerical results show that the anisotropic grading works with the modified method, and the convergence rate is optimal, just as with the standard method. 
This gives reason to investigate this situation in future research. 
Note for instance, that for $\nu = 1$ we have $\vec{f} = (0,0,\partial_z p)^T \in \vec{L}^2(\Omega)$, thus we can deduce using \cite[Lemma 3.1]{LinkeMerdonNeilan2019} that $\P(-\Delta \vec{u}) \in \vec{L}^2(\Omega)$ even for $\nu \neq 1$, in which case the first to components of $\vec{f}$ are not in $L^2(\Omega)$ anymore. Here $\P(\cdot)$ denotes the Helmholtz-Hodge projector, see \cite[Section 3]{LinkeMerdonNeilan2019}. 
The property $\P(-\Delta \vec{u})\in \vec{L}^2(\Omega)$ is the required regularity of the Laplacian of the velocity solution for the error analysis in \cite{LinkeMerdonNeilan2019}, which proves pressure-robust quasi-optimal estimates in low-regularity settings for the Stokes problem.

\Cref{fig:mesh_3d} shows the type of graded mesh used for this example, which satisfies the maximum angle condition, but not the regular vertex property. For a two dimensional domain $B=\{(r\cos(\phi),r\sin(\phi))\in \R^2: 0<r<1, 0<\phi<\omega\}$, a quasi-uniform mesh is created and graded towards the origin. The grading is done so that for a mesh size parameter $h$ and every triangle $T$ with diameter $h_T$ the relation
\begin{equation*}
	h_T \sim 	\begin{cases}
					h^{\nicefrac{1}{\mu}}, &\text{if } r_T = 0, \\
					h r_T^{1-\mu}, &\text{else},
				\end{cases}
\end{equation*}
is satisfied, where $r_T=\inf_{\vec{x}\in T}\{\dist(\vec{x}, \mathbf{0})\}$ and $\mu \in (0,1]$ is a grading parameter. The resulting, no longer quasi-uniform but still isotropic, mesh is then extended into the third dimension with a uniform mesh size $h_3 \sim h$. This pentahedral mesh is subsequently turned into a tetrahedral mesh by subdividing each prism into three tetrahedra, as shown in \Cref{fig:Pentahedron}. The procedure yields a mesh where the number of elements satisfies $N_{\text{elem}} \sim h^{-3}$ and is described in more detail in e.g. \cite{Apel1999,ApelNicaiseSchoberl2001,ApelNicaiseSchoberl2001:2}.

\begin{figure}[t]
	\centering
	\includegraphics[scale=1]{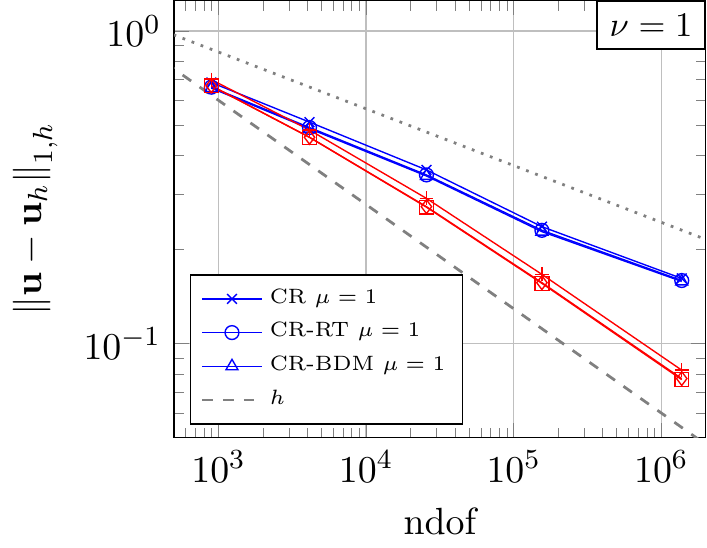}
	\hfill
	\includegraphics[scale=1]{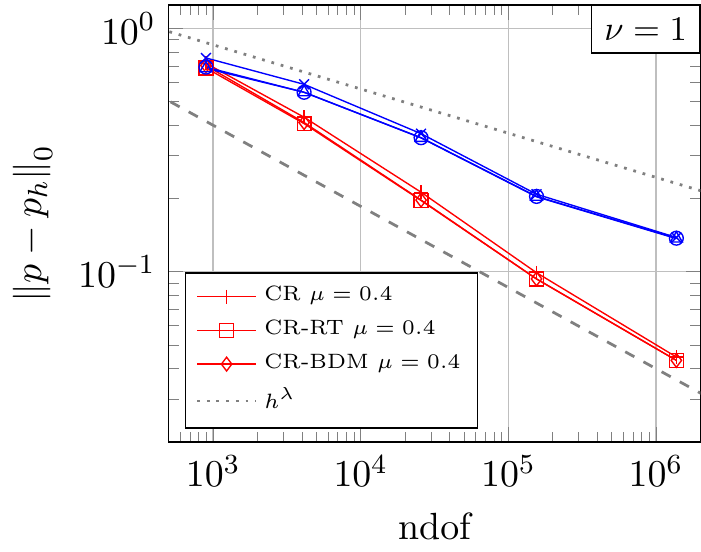}
	\vspace{1em}
	\includegraphics[scale=1]{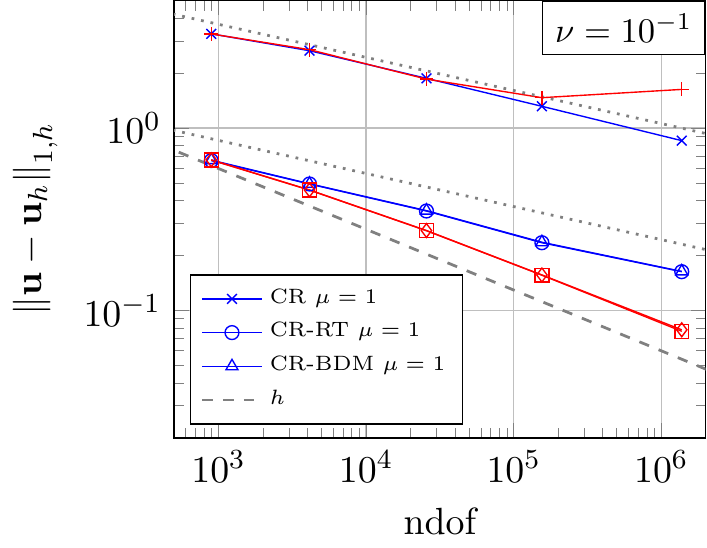}
		\hfill
		\includegraphics[scale=1]{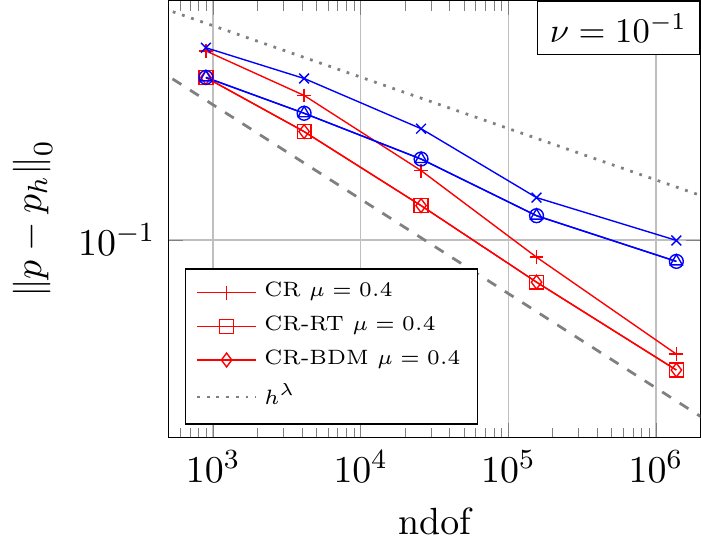}
	\caption{Energy and $L^2$ error of the discrete solution obtained with the standard Crouzeix--Raviart and modified Crouzeix--Raviart method}
	\label{fig:error_example2}
\end{figure}
The calculations were done with parameter values $\nu \in \{10^{-1},1\}$ and $\mu \in \{0.4, 1\}$. 
The results, see \Cref{fig:error_example2}, on the one hand corroborate the results from \cite{ApelNicaiseSchoberl2001} and on the other hand show that the recovery of the optimal convergence rates is also possible for the pressure-robust modified Crouzeix--Raviart method. 

From the results with the viscosity set to $\nu = 10^{-1}$, see \Cref{fig:error_example2}, it is clear that the modified method shows the pressure-robustness property also in these low regularity settings with anisotropic mesh grading, as the errors of the velocity are not influenced by the value of $\nu$.

\begin{remark}
	Due to a factor $r^{\lambda-2}$ arising in the first two components of the data function $\vec{f}$ for parameters $\nu \neq 1$, the numerical quadrature of the right hand side of the variational formulation has to be very accurate in order to achieve the presented numerical results.
\end{remark}

\printbibliography

\end{document}